\documentclass[10pt,a4paper]{amsart}

\usepackage{times,amsmath,amsbsy,amssymb,amscd,mathrsfs}
\usepackage{slashbox}
\usepackage{graphicx,subfigure,epstopdf,wrapfig,chemarrow}

\usepackage[ruled,linesnumbered]{algorithm2e}
\usepackage{multicol,multirow}
\usepackage{mathtools}
\usepackage[usenames,dvipsnames,svgnames,table]{xcolor}
\usepackage[all]{xy}
\usepackage{wrapfig}
\usepackage{tcolorbox}

 \def\p{\partial} 
\def\to{\rightarrow}
\definecolor{myBlue}{rgb}{0.0,0.0,0.55}
\usepackage[pdftex,colorlinks=true,citecolor=myBlue,linkcolor=myBlue]{hyperref}

\usepackage[hyperpageref]{backref}

\usepackage{comment,enumerate,multicol,xspace}

  \newcounter{mnote}
  \let\oldmarginpar\marginpar
    \renewcommand\marginpar[1]{\-\oldmarginpar[\raggedleft\footnotesize #1]%
    {\raggedright\footnotesize #1}}

\newtheorem{theorem}{Theorem}[section]

\newtheorem{remark}[theorem]{Remark}

\newcommand{\dx}{\,{\rm d}x}
\newcommand{\dd}{\,{\rm d}}

\newcommand{\bs}{\boldsymbol}

\newcommand{\vertiii}[1]{{\left\vert\kern-0.25ex\left\vert\kern-0.25ex\left\vert #1
    \right\vert\kern-0.25ex\right\vert\kern-0.25ex\right\vert}}

\begin{document}
\title{A Minimization Method for The Double-Well Energy Functional}
%\author{Long Chen}\date{\today}

\author{Qian Zhang}
\address{Beijing Institute for Scientific and Engineering Computing, Beijing University of Technology, Beijing 100124, China}
\email{liemeixiang@hotmail.com}

\author{Long Chen}
\address{Beijing Institute for Scientific and Engineering Computing, Beijing University of Technology, Beijing 100124, China and Department of Mathematics, University of California at Irvine, Irvine, CA 92697, U.S.A}
\email{chenlong@math.uci.edu}
\thanks{The work of Long Chen was supported by NSF Grant DMS-1418934 and in part by the Sea Poly Project of Beijing Overseas Talents.}

\author{Yifeng Xu}
\address{Department of Mathematics and Scientific Computing Key Laboratory of Shanghai Universities, Shanghai Normal University, Shanghai 200234, China}
\email{yfxu@shnu.edu.cn}

\thanks{
The work of Yifeng Xu was supported by National Natural Science Foundation of China (11201307), Ministry of Education of China (20123127120001) and Natural Science Foundation of Shanghai (17ZR1420800).}



\begin{abstract}
In this paper an iterative minimization method is proposed to approximate the minimizer to the double-well energy functional arising in the phase-field theory. The method is based on a quadratic functional posed over a nonempty closed convex set and is shown to be unconditionally energy stable. By the minimization approach, we also derive an variant of the first-order scheme for the Allen-Cahn equation, which has been constructed in the context of Invariant Energy Quadratization, and prove its unconditional energy stability.
\end{abstract}
\maketitle
%\tableofcontents

%\input{0_introduction}
%\input{1_method}
%\input{2_stability}
%\input{reference}

%\input{temp}

\section{Introduction}

The Allen-Cahn equation \cite{Allen1979} is a basic model, describing the evolution of a diffuse phase boundary concentrated in a small region of size $\epsilon$, in the phase-field theory.  Now it has been widely used in the simulation of interfacial dynamics of multi-component systems. From the mathematical perspective, this equation can also be viewed as an $L^2$-gradient flow of a free energy functional (also see \eqref{eq:orig_energy} below)
\[
E(u):=\int_\Omega \frac{1}{2}|\nabla u|^2 +\frac{1}{\epsilon^2}F(u) \dx,
\]
i.e.,
\[
    u_t-\Delta u+\frac{1}{\epsilon^2}F'(u)=0
\]
subject to $\frac{\partial u}{\partial \boldsymbol{n}}=0$ on ${\partial\Omega}$, where $\bs n$ is the unit outward normal on $\p\Omega$. It is easy to check that the solution of the Allen-Cahn equation satisfies the energy dissipation law for $E(u)$:
\begin{equation}\label{eq:disslaw}
    \frac{\dd}{\dd t} E(u)=-\|u_t\|^2_{L^2(\Omega)}\leq 0.
\end{equation}

In designing a numerical scheme for the Allen-Cahn equation, one naturally wish to preserve \eqref{eq:disslaw} in the discrete level. This leads to the so-called energy stability for the time discretization; see section 4 below for more detail. However, due to non-convexity of $F(u)$, a tiny time step size, e.g.,  $k\leq\epsilon^2$ for the double well potential $F(u)=\frac{1}{4}(u^2-1)^2$, is required to satisfy the energy stability if a fully implicit scheme is applied. To overcome this difficulty, a popular approach is the convex splitting scheme \cite{Elliott1993,Eyre1998,AKW2013,GLWW2014}. This scheme is energy stable without any stringent condition on the time step and a nonlinear system is solved at each time step. Another approach is the (stabilized) semi-implicit scheme \cite{Chen1998,Zhu1999,Shen2010}. This scheme allows a much larger time step size than the fully explicit scheme and is proved to be unconditionally energy stable for all stabilization constants no less than $\frac{L}{2}$ if the second derivative of the nonlinear potential $F(u)$ is bounded by some positive constant $L$.

Recently, an invariant energy quadratization (IEQ) approach is proposed in \cite{Yang2016} for the Allen-Cahn equation and the Cahn-Hilliard equation as a generalization of the augmented Lagrangian multiplier (ALM) method in \cite{Guillen2013}. IEQ can yield an unconditional energy stable scheme for a large class of free energies only if $F(u)\geq -C_0$, $C_0$ is a given positive constant. %the nonlinear potential $F(u)$ is bounded from below.
 Later, replacing $F(u)\geq -C_0$ by $\int_\Omega F(u)dx\geq -C_0$, Shen et al. developed a scalar auxiliary variable (SAV) approach  in \cite{Shen2018} to derive an unconditionally energy stable scheme. All these three approaches feature solving linear systems with variable or constant coefficients, but the numerical energy involved depends on the auxiliary variable and is not the original energy $E(u)$.

In this paper, we plan to deal with $E(u)$ directly by a
minimization approach. In particular, an iterative minimization
method to approximate $E(u)$ with the double well potential will
be presented. Assuming some previous $u^{n-1}$ is given, our
starting point is to approximate $\frac{1}{4}(u^2-1)^2$ by a
quadratic function. This in turn induces a quadratic functional
$E_Q(u;u^{n-1})$ approximating $E(u)$ at the neighbourhood of
$u^{n-1}$. Then we minimize $E_{Q}(u;u^{n-1})$ to get $u^{n}$; see Algorithm \ref{al:itercvx}. It is worth mentioning that the
minimization problem is constrained by the bound $|u|\leq 1$ a.e. in $\Omega$
so that Algorithm \ref{al:itercvx} is shown to be unconditionally
energy stable with respect to $E(u)$; see section 3. Moreover, we shall make use of this idea to derive an variant of the first-order scheme for the Allen-Cahn equation by IEQ in \cite{Yang2016} and provide a rigorous proof of the unconditional energy stability for the proposed scheme; see section 4. It should be pointed out that the energy used here is $E(u)$ not the modified one in \cite{Yang2016}.

The rest of this paper is organized as follows. In section 2, we propose an iterative minimization algorithm for $E(u)$ with $F(u)=\frac{1}{4}(u^2-1)^2$. Then the unconditional stability of the algorithm is proved in section 3. In section 4, we take the proposed minimization approach to the Allen-Cahn equation.

\section{An iterative convex minimization method}
We start with the following unconstrained minimization problem:
\begin{equation} \label{eq:original}
\min_{u\in H^1(\Omega)}\ E(u),
\end{equation}
where $\Omega \subset \Re^d$, $d=2,3$. The objective energy functional is defined by
\begin{eqnarray}\label{eq:orig_energy}
E(u):=\int_\Omega \frac{1}{2}|\nabla u|^2 +\frac{1}{\epsilon^2}F(u) \dx,
\end{eqnarray}
where $F: H^1(\Omega) \to \Re$ is a double well potential:
$$F(u)=\frac{1}{4}(u^2-1)^2.$$
Since $F(u)$ is non-convex, it is not easy to solve the minimization problem.

Motivated by ALM approach \cite{Guillen2013}, we introduce a new variable $p$, defined by
\begin{eqnarray*}
p=F^{\frac{1}{2}}(u)=\frac{1}{2}(u^2-1).
\end{eqnarray*}
Then we can rewrite \eqref{eq:original} as a constrained minimization problem with a quadratic energy functional:
\begin{equation} \label{eq:orig_constr_problem}
    \begin{aligned}
    \min\limits_{u\in H^1(\Omega),\ p\in L^2(\Omega)} &\int_\Omega \frac{1}{2}|\nabla u|^2 +\frac{1}{\epsilon^2}p^2 \dx \\
    \textrm {s.t.}~&p=\frac{1}{2}(u^2-1). %\nonumber
    \end{aligned}
\end{equation}
Although the objective functional of \eqref{eq:orig_constr_problem} is quadratic, we have to deal with a nonlinear constraint for $u$ and $p$. Thus this problem is still not easy to solve as \eqref{eq:original}.

We shall propose an iterative method to approximate \eqref{eq:original} based on the formulation \eqref{eq:orig_constr_problem}. At each iteration, a constrained minimization problem featuring a quadratic objective functional is solved. It is known that this kind of problems can be solved efficiently.

With an initial guess $u^0$ satisfying $\|u^0\|_{L^\infty(\Omega)} \leq 1$ given and $u^{n-1}$ standing for the solution at the $(n-1)$-th iteration. At the $n$-th iteration ($n\geq 1$), we approximate constraint $p=p(u)$ with its linear expansion at $u^{n-1}$. The linear approximation is denoted by $p_L(u;u^{n-1})$:
\begin{eqnarray*}
p_L(u;u^{n-1})&=&p(u^{n-1})+\langle p^\prime(u^{n-1}),u-u^{n-1}\rangle  \\
&=& \frac{1}{2}\big( (u^{n-1})^2-1 \big)+u^{n-1}(u-u^{n-1})  \\
&=& u^{n-1}u-\frac{1}{2}(u^{n-1})^2-\frac{1}{2}.
\end{eqnarray*}
Next we define a new quadratic energy functional, which is an approximation of the original energy $E(u)$ at point $u^{n-1}$, for the $n$-th iteration:
\begin{eqnarray}
\label{eq:EQ} E_Q(u;u^{n-1})&=&\int_\Omega \frac{1}{2}|\nabla u|^2+\frac{1}{\epsilon^2} p^2_L(u;u^{n-1}) \dx  \\
&=& \int_\Omega \frac{1}{2}|\nabla u|^2+\frac{1}{\epsilon^2}\left(u^{n-1}u-\frac{1}{2}(u^{n-1})^2-\frac{1}{2}\right)^2 \dx .
\end{eqnarray}
It is straight forward to verify that
\begin{enumerate}
\item $E_Q(u^{n-1};u^{n-1})=E(u^{n-1})$;
\item $E^\prime_Q(u^{n-1};u^{n-1})=E^{\prime}(u^{n-1})$;
\item $E^{\prime\prime}_Q(u^{n-1}) v, v \rangle=\|\nabla v\|^2+2(u^{n-1})^2\|v\|^2/\epsilon^2 \geq 0$ %=\|\nabla v\|^2+2(u^{n-1})^2\|v\|^2/\epsilon^2$
\end{enumerate}
That is $E_Q(u;u^{n-1})$ is a convex quadratic functional and a second order approximation of $E$ at $u^{n-1}$.

We solve the following constrained minimization problem at the $n$-th iteration:
\begin{eqnarray*}
\min_{u\in K} E_Q(u;u^{n-1}),
\end{eqnarray*}
where $K:=\left\{u\in H^1(\Omega) \left|~ |u|\leq 1~\mbox{a.e. in}~\Omega\right. \right\}$ is a nonempty closed convex subset of $H^1(\Omega)$. Since $E_{Q}$ is coercive and quadratic, the above problem has a solution in $K$ and can be solved efficiently~\cite{Tai2003}.

We are now in a position to introduce an iterative minimization method to solve minimization problem~\eqref{eq:original}.

\medskip

\IncMargin{1em}
\begin{algorithm}[H]

 Given $u^0\in K$ and set $n=1$\;

 Solve the constrained optimization problem
            \begin{equation}\label{ae-min}
                u^n=\textrm{arg}\min_{u\in K}E_Q(u;u^{n-1}) ;
            \end{equation} \\
 Set $n:=n+1$ and go to Step 2\;

    \caption{Iterative convex minimization method \label{al:itercvx}}
\end{algorithm}
\DecMargin{1em}

\section{Energy stability}
For energy minimization problems, we say that a method is {\em energy stable} if
\begin{eqnarray*}
E(u^n) \leq E(u^{n-1}),
\end{eqnarray*}
where $E(u)$ is the objective energy functional such as \eqref{eq:orig_energy}, $u^n$ is the result generated by the $n$-th iteration.
We will prove that Algorithm 1 is energy stable with respect to the original energy $E(u)$.

\begin{theorem}\label{th:energystable}
Let $u^n\in {\rm argmin}_{u\in K}\ E_Q(u;u^{n-1}), n=1,2,3,\ldots$ with $E_Q$ defined by \eqref{eq:EQ}. $E(u)$ is the  energy defined in \eqref{eq:orig_energy}, then
\begin{equation*}
E(u^n)\leq E(u^{n-1}).
\end{equation*}
\end{theorem}
\begin{proof}
Since $u^{n-1}\in K$ and $u^n\in {\rm arg}\min_{u\in K}\ E_Q(u;u^{n-1})$, the following inequality holds:
\begin{equation*}
E_Q(u^n;u^{n-1})=\min_{u\in K} E_Q(u;u^{n-1})\leq E_Q(u^{n-1};u^{n-1}) = E(u^{n-1}).
\end{equation*}
Now we prove that $E(u) \leq E_Q(u;u^{n-1})$ holds for all $u\in K$.
\begin{eqnarray*}\label{eq:diff_EQ}
E(u) - E_Q(u;u^{n-1}) &=& \int_\Omega p^2(u) - p^2_L(u;u^{n-1}) \dx \nonumber \\
&=& \int_\Omega \left[p(u)-p_L(u;u^{n-1})\right] \left[p(u)+p_L(u;u^{n-1})\right] \dx.
\end{eqnarray*}
It is easy to verify that $p(u) - p_L(u;u^{n-1})\geq 0$ due to the convexity of $p(u)$. We only need to prove that $p_L(u;u^{n-1})+p(u)\leq 0$ holds for any $|u|\leq 1$. That is,
\begin{eqnarray*}
p_L(u;u^{n-1})+p(u)
= \frac{1}{2}u^2+u^{n-1}u-\frac{1}{2}(u^{n-1})^2-1
\leq u^2-1 \leq 0, \quad \forall\ u\in K.
\end{eqnarray*}
where the last inequality holds since that $u$ satisfies the constraint $u\in K$.
Then
\begin{eqnarray*}
E(u^n) - E_Q(u^n;u^{n-1})\leq 0.
\end{eqnarray*}
Hence,
\begin{eqnarray*}
E(u^n) \leq E_Q(u^n;u^{n-1}) \leq E_Q(u^{n-1};u^{n-1}) = E(u^{n-1}).
\end{eqnarray*}
\end{proof}

%-----------------------------------------------------------------------------------------
\section{Application to the Allen-Cahn equation}%\label{app_Allen-Cahn}
The $L^2$-gradient flow of energy functional $E(u)$ defined by \eqref{eq:orig_energy} is the so-called Allen-Cahn equation:
\begin{equation}\label{eq:ac_eqn}
u_t-\Delta u+\frac{1}{\epsilon^2}f(u)=0,
\end{equation}
where $f(u)=F^\prime(u)=u(u^2-1)$.
In the previous section, we introduce Algorithm \ref{al:itercvx} to find a steady state of the Allen-Cahn equation \eqref{eq:ac_eqn} via an energy minimization problem \eqref{eq:original}. By this method, we can solve the problem \eqref{eq:original} efficiently and prove the energy stability. In this process, we treat the problem as a static one. In some scenarios, dynamics of the Allen-Cahn equation is more important.

Now we consider solving the Allen-Cahn equation \eqref{eq:ac_eqn}.
The first-order implicit scheme of \eqref{eq:ac_eqn} reads:
\begin{equation} \label{eq:discreteac}
\frac{u^n-u^{n-1}}{k}=\Delta u^n-\frac{1}{\epsilon^2}f(u^n),
\end{equation}
where $k>0$ is the time step size.

We say that a method is {\em unconditionally energy stable} if the energy stability is independent of time step size. That is,
\begin{eqnarray*}
E(u^{n-1}) \geq E(u^n)
\end{eqnarray*}
holds without any constraint on the time step size $k$. Otherwise we say that the method is {\em conditionally energy stable}.

It is well-known that the full implicit scheme is conditionally stable when $k\leq\epsilon^2 \ll 1$, and convex splitting is unconditionally stable with $f(u^n)$ in~\eqref{eq:discreteac} replaced by $f_+(u^n)-f_{-}(u^{n-1})$, where $f_+(u)=u^3$ and $f_{-}(u)=u$.

In order to apply Algorithm \ref{al:itercvx} to find a solution of \eqref{eq:discreteac}, following Xu et al \cite{Xu2016} we shall define an energy at point $u^{n-1}$:
\begin{eqnarray*}
J(u;u^{n-1})&=&E(u)+\frac{1}{2k}\|u-u^{n-1}\|^2_{L^2(\Omega)} \\
&=& \int_\Omega \frac{1}{2}|\nabla u|^2 + \frac{1}{\epsilon^2}F(u) \dx+\frac{1}{2k}\|u-u^{n-1}\|^2_{L^2(\Omega)} \\
&=& Q(u;u^{n-1})+\frac{1}{\epsilon^2}\int_\Omega F(u) \dx
\end{eqnarray*}
where $Q(u;u^{n-1})=\frac{1}{2}\|\nabla u\|^2_{L^2(\Omega)} +\frac{1}{2k}\|u-u^{n-1}\|^2_{L^2(\Omega)}$ is a quadratic functional of $u$.

We shall use Algorithm \ref{al:itercvx}, which requires to solve a constrained minimization problem as follows at the $n$-th iterations:
\begin{eqnarray}\label{eq:const_ac}
&\min_{u\in H^1(\Omega)}& J_Q(u;u^{n-1}), \nonumber\\
&\textrm{s.t.}& u\in K,
\end{eqnarray}
where the objective functional $J_Q(u;u^{n-1})$ is defined by
\begin{equation}\label{eq:JQ}
J_Q(u;u^{n-1}):=Q(u;u^{n-1})+\frac{1}{\epsilon^2}\int_\Omega p^2_L(u;u^{n-1}) \dx
\end{equation}
and $K$ is the closed convex set defined as in the previous section. Due to the strict convexity of $J_Q$, this problem admits a unique solution and can be solved efficiently.

Let $u^n$ be the minimizer of the unconstrained minimization problem
\begin{eqnarray*}
\min_{u\in H^1(\Omega)}\ J_Q(u;u^{n-1}),
\end{eqnarray*}
which satisfies
\[
    J_Q^\prime(u^n;u^{n-1})=0,
    \]
that is,
\begin{eqnarray}\label{eq:IEQ}
&&\frac{u^n-u^{n-1}}{k}=\Delta u^n-\frac{2}{\epsilon^2}(u^{n-1})^2u^n+\frac{1}{\epsilon^2}(u^{n-1})^3+\frac{1}{\epsilon^2}u^{n-1}. \label{eq:eqieq}
\end{eqnarray}
We note that \eqref{eq:eqieq} is the same as the first-order scheme of the IEQ approach \cite{Guillen2013,Yang2016}.

Here we provide a different interpretation of IEQ from the perspective of the energy minimization. The constraint $\|u\|_{L^\infty(\Omega)}\leq 1$ is added in order to prove the scheme is unconditional stable. Such a priori bound can be proved for a weak solution of the Allen-Cahn equation (see \cite{Feng2007}) and can be established for a modified fully implicit scheme proposed in Xu, Li and Wu~\cite{Xu2016}. However, we can not prove it for the IEQ iteration \eqref{eq:IEQ} and thus explicitly impose it as a constraint.

\begin{theorem}
Let $u^n={\rm argmin}_{u\in K}\ J_Q(u;u^{n-1})$, $n=1,2,3,\ldots$ with $J_Q$ defined in \eqref{eq:JQ}. $E(u)$ is the energy functional of the Allen-Cahn equation, defined by \eqref{eq:orig_energy}, then the energy is unconditionally stable, that is,
\begin{eqnarray*}
E(u^n)\leq E(u^{n-1}).
\end{eqnarray*}
\end{theorem}
\begin{proof}
The proof is similar to Theorem \ref{th:energystable}. By definition,
\begin{eqnarray*}
J_Q(u^{n-1};u^{n-1})
=\int_\Omega \frac{1}{2}|\nabla u^{n-1}|^2 +\frac{1}{\epsilon^2} p^2(u^{n-1}) \dx
=E(u^{n-1}).
\end{eqnarray*}
Since $u^{n-1}\in K$ and $u^n \in {\rm argmin}_{u\in K} J_Q(u;u^{n-1})$, we have
$$J_Q(u^{n};u^{n-1}) = \min_{u\in K} J_Q(u;u^{n-1}) \leq J_Q(u^{n-1};u^{n-1}) = E(u^{n-1}).$$
For any $u\in K$, similar arguments to those in the proof of Theorem \ref{th:energystable} imply that
\begin{eqnarray*}
J_Q(u;u^{n-1})-E(u)
&=& \frac{1}{\epsilon^2}\int_\Omega p^2_L(u;u^{n-1})-p^2(u) \dx+\frac{1}{2k}\|u^{n-1}-u\|^2_{L^2(\Omega)} \\
&\geq& \frac{1}{\epsilon^2}\int_\Omega p^2_L(u;u^{n-1})-p^2(u) \dx \\
&\geq&0
\end{eqnarray*}
Collecting the above two estimates with $u=u^n$, we then obtain the conclusion.
\end{proof}

\begin{remark}\rm
    In order to ensure the unconditional stability of $E(u)$, the constraint $u\in K$ is incorporated in \eqref{eq:ac_eqn}, without which, we cannot prove
    \begin{eqnarray*}
        E(u^n)\leq E(u^{n-1}).
\end{eqnarray*}
The unconditional energy stability in \cite{Yang2016} is indeed the inequality
$$
J_Q(u^{n};u^{n-1}) \leq J_Q(u^{n-1};u^{n-1}),
$$
which corresponds to a modified energy not the original one.
\end{remark}

%\bibliographystyle{abbrv}
 %\bibliography{/Dropbox/Math/biblib/library}

\begin{thebibliography}{999}

\bibitem{Allen1979}
{S. Allen, J. W. Cahn}. {A Microscopic Theory for Antiphase Boundary Motion and Its Application to Antiphase Domain Coarsening}, {Acta Metall. Mater.}, 27: 1084-1095, 1979.

\bibitem{AKW2013}
{A. Aristotelous, O. Karakashian, S. M. Wise}, {A mixed discontinuous Galerkin, convex splitting scheme for a modified Cahn-Hilliard equation and an efficient nonlinear multigrid solver}, {Discrete and Continuous Dynamical System Series B}, 18: 2211-2238, 2013.

\bibitem{Chen1998}
{L. Chen, J. Shen}. {Applications of semi-implicit Fourier-spectral method to phase field equations}, {Computer Physics Communications}. 108: 147-158, 1998.

\bibitem{Elliott1993}
{C. M. Elliott, A.M. Stuart}. {The global dynamics of discrete semilinear parabolic equations}, {SIAM Journal on Numerical Analysis} 30: 1622-1663, 1993. % https://doi.org/10.1137/0730084.

\bibitem{Eyre1998}
{D. J. Eyre}. {Unconditionally gradient stable time marching the Cahn-Hilliard equation}, in: Computational and Mathematical Models of Microstructural Evolution, {J. W. bullard, R. Kalias, M. Stoneham, L. Q. Chen} eds., Mater. Res. Soc. Symp. Proc., vol.529, MRS, Warrendale, PA, 1998, 39-46.

\bibitem{Feng2007}
{X. Feng, Y. He, and C. Liu}. {Analysis of finite element approximations of a phase field model for two-phase fluids}. {Math. Comp.}, 76 (258): 539-571, 2007.

\bibitem{GLWW2014}
{Z. Guan, J. S. Lowengrub, C. Wang, S. M. Wise}, {Second order convex splitting schemes for periodic nonlocal
Cahn-Hilliard and Allen-Cahn equations}, {Journal of Computational Physics}, 277: 48-71, 2014.

\bibitem{Guillen2013}
{F. Guill$\acute{\rm e}$n-Gonz$\acute{\rm a}$lez, G. Tierra}. {On linear schemes for a Cahn-Hilliard diffuse interface model}. {Journal Computational Physics}, 234: 140-171, 2013.

%\bibitem{Glow1984}
%{R. Glowinski}. {Numerical Methods for Nonlinear Variational
%Problems}, Springer-Verlag, New York, 1984.

\bibitem{Tai2003} {X. Tai}.
{Rate of Convergence for Some Constraint Decomposition Methods
for Nonlinear Variational Inequalities}.
{Numerische Mathematik}, 93: 755-786, 2003.

\bibitem{Shen2018}{J. Shen, J. Xu, J. Yang}.
{The Scalar Auxiliary Variable (SAV) Approach for Gradient Flows}.
{Journal of Computational Physics}, 353: 407-416, 2018.

\bibitem{Shen2010}
{J. Shen, X. Yang}. {Numerical approximations of Allen-Cahn and Cahn-Hilliard equations}. {Discrete and Continuous Dynamical System, Series A}, 28: 1669-1691, 2010.

\bibitem{Xu2016}
{J. Xu, Y. Li, S. Wu}. {On the accuracy of partially implicit schemes for phase field modeling}. preprint, 2016, available at arXiv: 1604.05402v3.

\bibitem{Yang2016}
{X. Yang}. {Linear, first and second-order, unconditionally energy stable numerical schemes for the phase field model of homopolymer blends}. {Journal of Computational Physics}, 327: 294-316, 2016. %https://doi.org/10.1016/j.jcp.2016.09.029.

\bibitem{Zhu1999}
{J. Zhu, L.-Q. Chen, J. Shen, V. Tikare}. {Coarsening kinetics from a variable-mobility Cahn-Hilliard equation: Application of a semi-implicit Fourier spectral method}, {Physical Review E}, 60: 3564-3572, 1999.

\end{thebibliography}
\end{document}